\documentclass[12pt]{amsart}

\usepackage{amsthm,amssymb,amstext,amscd,amsfonts,amsbsy,amsxtra,latexsym,amsmath,
color,verbatim
}
\usepackage[english]{babel}
\usepackage[latin1]{inputenc}
\usepackage{hyperref}
\usepackage{mathrsfs}
\allowdisplaybreaks

\newcommand{\field}[1]{\mathbb{#1}}

\newcommand{\Z}{\field{Z}}
\newcommand{\R}{\field{R}}

\renewcommand{\H}{\mathbb{H}}

\newtheorem{thm}{Theorem}

\newtheorem{prop}{Proposition}

\theoremstyle{remark}
\newtheorem{rmk}{Remark}
\newtheorem{ex}{Example}

\theoremstyle{definition}

\title{Linear Incongruences for Generalized Eta-Quotients}
\author{Steffen L\"obrich }

\begin{document}

\maketitle

\begin{abstract}
For a given generalized eta-quotient, we show that linear progressions whose residues fulfill certain quadratic equations do not give rise to a linear congruence modulo any prime. This recovers known results for classical eta-quotients, especially the partition function, but also yields linear incongruences for more general weakly holomorphic modular forms like the Rogers-Ramanujan functions.
\end{abstract}

\section{Introduction and Statement of Results}
Ever since Ramanujan established his famous linear congrucences 
\begin{equation}\label{ram}
\begin{split}
p(5n+4)&\equiv 0 \pmod{5}\\
p(7n+5)&\equiv 0 \pmod{7} \\
p(11n+6)&\equiv 0\pmod{11}
\end{split}
\end{equation}
for the partition function $p(n)$, the challenge of proving and generalizing them triggered a vast amount of research. For instance, Ono \cite{ono} found analogues of \eqref{ram} for every modulus coprime to $6$. See also \cite{ao} and the sources contained therein for further results. However, recent work of Radu \cite{ra} proved that there are no linear congruences of $p(n)$ modulo $2$ and $3$, affirming a famous conjecture of Subbarao. The main ingredients of his proof are skillful computations and the {\it $q$-expansion principle} due to Deligne and Rapoport \cite{dr}. Adapting the methods of Radu's proof,   Ahlgren and Kim \cite{ak} showed analogous results for the mock theta functions $f(q)$ and $\omega(q)$, as well as for certain classes of weakly holomorphic modular forms, including (classical) eta-quotients. In this paper, we extend their approach to generalized eta-quotients.\\

These functions are defined as follows: For $\delta\in\Z^+$ and a residue class $g\pmod{\delta}$, we set
$$
\eta_{\delta, g}(z):=q^{\frac{\delta}{2} P_2\left(\frac{g}{\delta}\right)}\prod_{m>0\atop m\equiv  g\pmod{\delta}}\left(1-q^m\right)\prod_{m>0\atop m\equiv  -g\pmod{\delta}}\left(1-q^m\right),
$$
where $z\in\H$ and $q:=e^{2\pi iz}$ throughout. Here, for $x\in\R$ and $\{x\}:=x-\left\lfloor x\right\rfloor$, we let
$$
P_2(x):=\{x\}^2 -\{x\}+\frac16
$$
be the \emph{second Bernoulli function}. \\

Note that if 
$$
\eta(z):=q^{1/24}\prod_{m> 0}(1-q^m)
$$
denotes the usual Dedekind eta-function, then 
$$
\eta_{\delta, 0}(z)=\eta(\delta z)^2\qquad\text{ and }\qquad\eta_{\delta, \frac{\delta}{2}}(z)=\frac{\eta\left(\frac{\delta}{2}z\right)^2}{ \eta(\delta z)^2}.
$$ 
Furthermore, for $g\notin\left\{ 0,\frac{\delta}{2}\right\}$ we have
$$
\eta_{\delta, g}(z)^{-1} = q^{-\frac{\delta}{2} P_2\left(\frac{g}{\delta}\right)}\sum_{n\geq 0} p_{\delta, g}(n)q^n,
$$
where $p_{\delta, g}(n)$ denotes the number of partitions of $n$ with all parts congruent to $\pm g\pmod{\delta}$.\\

For $\delta=5$, these functions occur in the well-known {\it Rogers-Ramanujan identities}, which state that
$$
q^{\frac{1}{60}}\eta_{5,1}^{-1}(z)=\sum_{n\geq 0}\frac{q^{n^2}}{(q;q)_n}=1+q+q^2+q^3+2q^4+2q^5+3q^6+\dots
$$
and 
$$
q^{-\frac{11}{60}}\eta_{5,2}^{-1}(z)=\sum_{n\geq 0}\frac{q^{n^2+n}}{(q;q)_n}=1+q^2+q^3+q^4+q^5+2q^6+\dots,
$$
where $(q;q)_n:=\prod_{j=1}^{n}(1-q^j)$.\\

For $N\in\Z^+$, a residue class $a\pmod{N}$, let $r:=(r_{\delta, g})_{\delta|N, g\pmod{\delta}}$ be a tuple of half-integers, indexed by the divisors of $N$ and their residue classes, with $r_{\delta, g}\in\Z$ unless $g=0$ or $g=\frac{\delta}{2}$. In this paper, we study {\it generalized eta-quotients} of the form
$$
H_{r} (z):=\prod_{\delta |N\atop g\pmod{\delta}}\eta_{\delta, g}(z)^{r_{\delta,g}}=:q^{P(r)}\sum_{n\geq 0}c_r (n)q^n,
$$
where 
$$
P(r):=\frac12\sum_{\delta |N\atop g\pmod{\delta}}\delta r_{\delta, g}P_2\left(\frac{g}{\delta}\right).
$$
Note that the denominator of $P(r)$ divides $12N$. \\

For every modulus $m\in\Z ^+$ and residue class $t\pmod{m}$, we give conditions on prime numbers $p$ that guarantee that the linear progression $t\pmod{m}$ does not satisfy a linear congruence $\operatorname{mod} p$ for the generalized eta-quotient $H_{r}$. Here, for any residue class $a\pmod{N}$, we denote by $r_a$ the tuple $(r_{\delta, ag})_{\delta|N, g\pmod{\delta}}$.
\begin{thm}\label{thm}
Let $m\in\Z^+$ and $t\in\{0,\dots, m-1\}$. For $a,d\in\Z$ with $ad\equiv 1\pmod{24Nm}$, let $n$ be the smallest nonnegative integer for which
$$
d^2(n+P(r_a))- P(r)\equiv t\pmod{m}.
$$
Then for every prime $p$ not dividing $c_{r_a}(n)$, we have
$$
\sum_{n\geq 0}c_r(mn+t)q^n\not\equiv 0\pmod{p}.
$$
\end{thm}

\begin{rmk}
Since we always have $c_r(0)=1$, the linear incongruence is satisfied for any prime $p$ if 
$$
d^2P(r_a)- P(r)\equiv t\pmod{m}.
$$

\end{rmk}

\begin{rmk}
By work of Ahlgren and Boylan \cite{ab}, if the conditions of Theorem \ref{thm} are satisfied, we even have that
$$
\#\left\{n\leq X: \, c_r(mn+t)\not\equiv 0\pmod{p}\right\}\gg_{p,r,m,t,K} \frac{\sqrt{X}}{\log X}\left(\log\log X\right)^K
$$
for every positive integer $K$.
\end{rmk}

Theorem \ref{thm} has several immediate applications.
\begin{ex}
Let $N=a=1$ and $r=-\frac12$. Then 
$$
H_{-\frac12}(z)=\eta^{-1}(z)=q^{-\frac{1}{24}}\sum_{n\geq 0}p(n)q^n.
$$
Since $p(0)=p(1)=1$, Theorem \ref{thm} then implies that
$$
\sum_{n\geq 0}p(mn+t)q^n\not\equiv 0\pmod{\ell}
$$
for every prime $\ell$ if there is a $d$ coprime to $6m$ with 
$$
t\equiv\frac{1-d^2}{24}\pmod{m} \quad\text{ or }\quad t\equiv\frac{1+23d^2}{24}\pmod{m}.
$$

Now assume that $\ell\geq 5$ is prime with $\left(\frac{-23}{\ell}\right)=-1$. Then the classes $d^2\pmod{\ell}$ and $-23d^2\pmod{\ell}$ together run over all residue classes except for $0$ as $d$ runs over residue classes coprime to $\ell$. Since $(\ell,24)=1$, the classes $\frac{1-d^2}{24}$ and $\frac{1+23d^2}{24}$ cover every residue class modulo $\ell$ except for $\frac{1-\ell^2}{24}$. It follows that we can only have a linear congruence
$$
\sum_{n\geq 0}p(\ell n+t)q^n\equiv 0\pmod{\ell}
$$
if $t\equiv\frac{1-\ell^2}{24}\pmod{\ell}$. This result was shown by Kiming and Olsson for every prime $\ell$ \cite{ko}. In particular, for $\ell\in\{5,7,11\}$, this implies that the residues in \eqref{ram} are the only ones for which such a congruence can hold. 
\end{ex}

\begin{ex}
More generally, Theorem \ref{thm} specializes to classical eta-quotients if $r_{\delta, g}=0$ for $g\neq 0$. Then we have $P(r_a)=\frac{1}{12}\sum_{\delta |N}\delta r_{\delta}$ for all $a$. Since we always have $c_a(0)=1$, we obtain that for every prime $p$, we have
$$
\sum_{n\geq 0}c_r(mn+t)q^n\not\equiv 0\pmod{p}\qquad\text{ if }\qquad t\equiv\frac{d^2-1}{12}\sum_{\delta |N}\delta r_{\delta}\pmod{m}
$$ 
for some $d$ coprime to $6Nm$.
\end{ex}

\begin{ex}
Another interesting example are partitions occurring in {\it Schur's Theorem} \cite{shu}. These are given by 
$$
q^{\frac{1}{12}}\eta_{6,1}^{-1}(z)=\sum_{n\geq 0}p_{6,1}(n)q^n = 1+q+q^2+q^3+q^4+2q^5+2q^6+\dots
$$ 
with $N=6$, $r_{6,1}=-1$ and $r_{\delta, g}=0$ otherwise, and $P(r_a)=-\frac{1}{12}$ for every $a$ coprime to $6$.
Thus Theorem 1 implies that 
$$
\sum_{n\geq 0}p_{6,1}(mn+t)q^n\not\equiv 0\pmod{p}
$$
for any prime $p$ if there is a $d$ coprime to $6m$ and $j\in\{-1,11,23,35,47\}$ with
$$
t\equiv\frac{1+jd^2}{12}\pmod{m}. 
$$
As in Example 1, if $\ell\geq 5$ is a prime with at least one of $\left(\frac{-11}{\ell}\right)$, $\left(\frac{-23}{\ell}\right)$, $\left(\frac{-35}{\ell}\right)$, or $\left(\frac{-47}{\ell}\right)$ equal to $-1$, then there can only be a linear congruence if $t\equiv \frac{1-\ell^2}{12}\pmod{\ell}$.
\end{ex}

\begin{ex}
Now we take a closer look at the Rogers-Ramanujan functions $\eta_{5,1}^{-1}$ and $\eta_{5,2}^{-1}$. If $H_{r_1}=H_{r_4}=\eta_{5,1}^{-1}$, then we have $N=5$, $r_{5,1}=-1$, $r_{5,2}=0$, $H_{r_2}=H_{r_3}=\eta_{5,2}^{-1}$, and 
$$
P(r_a)=\begin{cases}
-\frac{1}{60} & \text{ if $a\equiv 1$ or $4 \pmod{5}$}, \\
\frac{11}{60}& \text{ if $a\equiv 2$ or $3 \pmod{5}$}.
\end{cases}
$$
Hence Theorem 1 states that $\sum_{n\geq 0}p_{5,1}(mn+t)q^n\not\equiv 0\pmod{p}$ for every prime $p$, if 
\begin{align*}
t &\equiv nd^2 +\frac{1-d^2}{60}\pmod{m} \text{ for $n\in\{0,1,2,3\}$ and $d\equiv 1, 4\pmod{5}$ coprime to $6m$} 
\end{align*}
or
\begin{align*}
t &\equiv nd^2 +\frac{11d^2+1}{60}\pmod{m}\text{ for $n\in\{0,2,3,4,5\}$ and $d\equiv 2, 3\pmod{5}$ coprime to $6m$.} 
\end{align*}

If we switch the roles of $\eta_{5,1}^{-1}$ and $\eta_{5,2}^{-1}$, we obtain that $\sum_{n\geq 0}p_{5,2}(mn+t)q^n\not\equiv 0\pmod{p}$ for every prime $p$, if 
\begin{align*}
t & \equiv nd^2 -\frac{d^2+11}{60}\pmod{m} \text{ for $n\in\{0,1,2,3\}$ and $d\equiv 2,3 \pmod{5}$ coprime to $6m$} 
\end{align*}
or
\begin{align*}
t &\equiv nd^2 +\frac{11(d^2-1)}{60}\pmod{m}\text{ for $n\in\{0,2,3,4,5\}$ and $d\equiv 1, 4\pmod{5}$ coprime to $6m$.}
\end{align*}

In contrast, applying work of Gordon \cite{go}, Hirschhorn \cite{hh} found linear congruences $\pmod{2}$ for $p_{5,1}$ and $p_{5,2}$. For example, Theorem 3 of \cite{hh} states that
$$
p_{5,1}(98n+t)\equiv 0 \pmod{2}
$$
for $t\in\{23,37,51,65,79,93\}$ and 
$$
p_{5,2}(98n+t)\equiv 0 \pmod{2}
$$
for $t\in\{6,20,34,62,76,90\}$.
The above discussion precludes all the other residues $\pmod{98}$ except for $t\in\{9,16,58,72,86\}$ resp. $t\in\{13,27,48,55,97\}$ from satisfying these congruences.
\end{ex}

The paper is organized as follows: In Section 2 we define generalized eta-quotients and study their transformation behavior under $\Gamma_0(12N)$, slightly adapting a result of Robins \cite{rob}. This will lead to modularity properties for the functions $H_{m,r,t}$ whose Fourier coefficients are given by those of $H_{r}$ on the arithmetic progression $t\pmod{m}$. In Section 3 we prove Theorem \ref{thm} using the $q$-expansion principle.

\section*{Acknowledgements}
We thank Ken Ono for suggesting this project and for his advice and support. We also thank Min-Joo Jang, Wenjun Ma, and the referee for helpful discussions, the Fulbright Commission for their generous support, and the Department of Mathematics and Computer Science at Emory University for their hospitality.  

\section{Transformation properties of Eta-Quotients}

We begin by studying modularity properites of $\eta_{\delta, g}$. For $A=\left(\begin{smallmatrix} a&b\\c&d\end{smallmatrix}\right)\in\Gamma_0(\delta)$ we define $\mu_{A,g,\delta}$ by 
$$
\eta_{\delta, g}\left( Az\right)=e\left(\mu_{A,g,\delta}\right)j(A,z)^{\delta_{g,0}}\eta_{\delta,ag}(z),
$$
where $j(A,z):=cz+d$ and $e(w):=e^{2\pi iw}$ throughout. An analogue of the following proposition for the subgroup $\Gamma_1(\delta)$ was shown in Theorem 2 of \cite{rob}.

\begin{prop} \label{muagd}
For $A=\left(\begin{smallmatrix} a&b\\c&d\end{smallmatrix}\right)\in\Gamma_0(12\delta)$ we have
$$
\mu_{A,g,\delta}\equiv \frac12 db\delta P_2\left(\frac{ag}{\delta}\right)-\frac{a-1}{4}+\frac{1}{2}\left\lfloor \frac{ag}{\delta}\right\rfloor\pmod{1}.
$$
\end{prop}

\begin{proof}
An equation on p.~122 of \cite{rob} states that (note the different normalization of $\mu_{A,g,\delta}$)
$$
\mu_{A,g,\delta}=\sum_{\mu=1}^{a-1}\left(\left(\frac{\mu }{a}\right)\right)\left(\left(\frac{c}{\delta}\frac{\mu }{a}+\frac{g}{\delta}\right)\right)+\frac{\delta b}{2a}P_2\left(\frac{ag}{\delta}\right)-\frac{c}{12\delta a}
$$
with 
$$
\left(\left(x\right)\right):=\begin{cases}
\{x\}-\frac12 & \text{ if $x\neq 0$,}\\
0 & \text{if $x=0$.}
\end{cases}
$$
By Eqn.~(34) of \cite{sch}, Ch.~VIII \S 4,  the denominator of $\mu_{A,g,\delta}$ divides $12\delta$. This implies that for $A\in\Gamma_0(12\delta)$ we have, using that $ad\equiv 1\pmod{12\delta}$, 
$$
\mu_{A,g,\delta}\equiv ad \sum_{\mu=1}^{a-1}\left(\left(\frac{\mu }{a}\right)\right)\left(\left(\frac{c}{\delta}\frac{\mu }{a}+\frac{g}{\delta}\right)\right)+\frac{\delta db}{2}P_2\left(\frac{ag}{\delta}\right)\pmod{1}.
$$
We compute
\begin{multline*}
ad\sum_{\mu=1}^{a-1}\left(\left(\frac{\mu }{a}\right)\right)\left(\left(\frac{c}{\delta}\frac{\mu }{a}+\frac{g}{\delta}\right)\right)
=d\sum_{\mu=1}^{a-1}\left(\mu-\frac{a}{2}\right)\left(\left(\frac{c}{\delta}\frac{\mu }{a}+\frac{g}{\delta}\right)\right)\\
=d\sum_{\mu=1}^{a-1}\mu\left(\left(\frac{c}{\delta}\frac{\mu }{a}+\frac{g}{\delta}\right)\right)-\frac{ad}{2}\sum_{\mu=1}^{a-1}\left(\left(\frac{\mu }{a}+\frac{g}{\delta}\right)\right).
\end{multline*}
Now 
$$
d\sum_{\mu=1}^{a-1}\mu\left(\left(\frac{c}{\delta}\frac{\mu }{a}+\frac{g}{\delta}\right)\right)\equiv d\sum_{\mu=1}^{a-1}\mu\left(\frac{c}{\delta}\frac{\mu }{a}+\frac{g}{\delta}-\frac12\right)
\equiv \frac{a-1}{2}\left(\frac{g}{\delta}-\frac12\right) \pmod{1}
$$
and 
\begin{multline*}
\frac{ad}{2}\sum_{\mu=1}^{a-1}\left(\left(\frac{\mu }{a}+\frac{g}{\delta}\right)\right)\equiv \frac{1}{2}\sum_{\mu=1}^{a-1}\left(\frac{\mu }{a}+\frac{g}{\delta}-\left\lfloor\frac{\mu }{a}+\frac{g}{\delta}\right]-\frac12\right)\\
\equiv \frac{a-1}{4}+\frac{a-1}{2}\left(\frac{g}{\delta}-\frac12\right) -\frac{1}{2}\left\lfloor\frac{ag}{\delta}\right\rfloor \pmod{1},
\end{multline*}
using that $\sum_{\mu=0}^{a-1}\left\lfloor\frac{\mu}{a}+x\right\rfloor=\lfloor ax\rfloor$.

\end{proof}
Let 
\begin{equation}\label{Hrmt}
H_{r,m,t}(z):=\frac{1}{m}\sum_{\lambda\pmod{m}}e\left(-\frac{\lambda}{m}(t+P(r))\right)H_{r}\left(\begin{pmatrix}
1 &  \lambda \\ & m
\end{pmatrix}z\right) ,
\end{equation}
so that
$$
H_{r,m,t}(z)= q^{\frac{t+P(r)}{m}}\sum_{n\geq 0}c_r(mn+t)q^n
$$
and for every $\lambda\pmod{m}$ and $\left(\begin{smallmatrix} a&b\\c&d\end{smallmatrix}\right)\in\Gamma_0(m)$ choose $\lambda'$ with
$$
a\lambda' \equiv b+d\lambda\pmod{m}.
$$
Note that $\lambda'$ runs over all residue classes modulo $m$ with $\lambda$. \\

Moreover, let $k:=\sum_{\delta |N} r_{\delta, 0}$, so that $k$ is the weight of $H_r$.

\begin{prop}\label{slash}
For $A\in \Gamma_0(24Nm)$, we have
\begin{equation*}
H_{r, m, t}(Az)=j(A,z)^k\frac{\zeta}{m}\sum_{\lambda\pmod{m}}e\left(\frac{\lambda}{m}\left(d^2 P(r_a)-P(r)-t\right)-\frac{\lambda'}{m} P(r_a)\right)H_{r_a}\left(\begin{pmatrix}
1 &  \lambda' \\ & m
\end{pmatrix}z\right),
\end{equation*}
where $\zeta$ is a $24Nm$-th root of unity depending on $r$, $m$, and $A$. In particular, $H_{r,m,t}^{24Nm}$ is a weakly holomorphic modular form of weight $24Nmk$ for $\Gamma_1(24Nm)$, i.e.~a meromorphic modular form whose poles are supported at the cusps. 
\end{prop}

\begin{proof}
Let
$$
A_\lambda:=\begin{pmatrix}
a+c\lambda &  \frac{1}{m}\left(b+d\lambda-\lambda'(a+c\lambda)\right)\\
mc & d-c\lambda'
\end{pmatrix},
$$
so that
$$
\begin{pmatrix}
1 &  \lambda \\ & m
\end{pmatrix}A=
A_\lambda\begin{pmatrix}
1 &  \lambda' \\ & m
\end{pmatrix}
$$

Then for $A\in \Gamma_0(24Nm)$, we have by Proposition \ref{muagd}

\begin{multline*}
\eta_{\delta,g}\left(\begin{pmatrix}
1 &  \lambda \\ & m
\end{pmatrix}Az\right)^{r_{\delta, g}} = \eta_{\delta,g}\left(A_\lambda\begin{pmatrix}
1 &  \lambda' \\ & m
\end{pmatrix}z\right)^{r_{\delta, g}}\\
=e\left({r_{\delta, g}}\mu_{A_\lambda, g, \delta}\right)j\left(A_\lambda,\begin{pmatrix}
1 &  \lambda' \\ & m
\end{pmatrix}z\right)^{\delta_{g,0}r_{\delta, g}}\eta_{\delta,ag}\left(\begin{pmatrix}
1 &  \lambda' \\ & m
\end{pmatrix}z\right)^{r_{\delta, g}}\\
=\zeta_0 e\left(\frac{{r_{\delta, g}}\delta}{2}P_2\left(\frac{ag}{\delta}\right)(d-c\lambda')  \frac{1}{m}\left(b+d\lambda-\lambda'(a+c\lambda)\right)\right)j(A,z)^{\delta_{g,0}r_{\delta, g}}\eta_{\delta,ag}\left(\begin{pmatrix}
1 &  \lambda' \\ & m
\end{pmatrix}z\right)^{r_{\delta, g}}\\
=\zeta_0 e\left(\frac{{r_{\delta, g}}\delta}{2m}P_2\left(\frac{ag}{\delta}\right)\left(db+d^2\lambda-\lambda'\right)\right)j(A,z)^{\delta_{g,0}r_{\delta, g}}\eta_{\delta,ag}\left(\begin{pmatrix}
1 &  \lambda' \\ & m
\end{pmatrix}z\right)^{r_{\delta, g}},
\end{multline*}
where $\zeta_0$ is a fourth root of unity depending on $r_{\delta, g}$ and $A$. 
Thus,
$$
H_{r}\left(\begin{pmatrix}
1 &  \lambda \\ & m
\end{pmatrix}Az\right) =\zeta j(A,z)^k  e\left(\frac{P(r_a)}{m}\left(d^2\lambda-\lambda'\right)\right) H_{r_a}\left(\begin{pmatrix}
1 &  \lambda' \\ & m
\end{pmatrix}z\right).
$$
Together with \eqref{Hrmt}, this yields the formula in the proposition.\\

Moreover, note that for $A\in\Gamma_1(24Nm)$, we have $\lambda'\equiv \lambda +b\pmod{m}$ and
\begin{multline*}
H_{r, m, t}(Az)=j(A,z)^k\frac{\zeta}{m}\sum_{\lambda'\pmod{m}}e\left(\frac{\lambda'-b}{m}\left((d^2 -1)P(r)-t\right)-\frac{\lambda'}{m} P(r)\right)H_{r}\left(\begin{pmatrix}
1 &  \lambda' \\ & m
\end{pmatrix}z\right)\\
=j(A,z)^k\frac{\zeta_1}{m}\sum_{\lambda'\pmod{m}}e\left(-\frac{\lambda'}{m}(t+ P(r))\right)H_{r}\left(\begin{pmatrix}
1 &  \lambda' \\ & m
\end{pmatrix}z\right)=\zeta_1  j(A,z)^k H_{r, m, t}(z)
\end{multline*}
with $\zeta_1:=e\left(\frac{bt}{m}\right)\zeta$. Since $\zeta_1$ is a $24Nm$-th root of unity, we conclude that $H_{r,m,t}^{24Nm}$ is a weakly holomorphic modular form of weight $24Nmk$ for $\Gamma_1(24Nm)$.
\end{proof}

\section{Proof of Theorem \ref{thm} }
\begin{proof}
For $j$ large enough, we have that $H_{r,m,t}^{24Nm}\Delta^j$ is a holomorphic modular form for $\Gamma_1(24Nm)$. Let $A=\left(\begin{smallmatrix}
a &  b \\ c & d
\end{smallmatrix}\right)\in \Gamma_0(24Nm)$ and let $|_k$ denote the Petersson slash-operator, i.e. $(f|_k A)(z):=j(A,z)^{-k}f(Az)$. Then since
$$
H_{r_a}\left(\begin{pmatrix}
1 &  \lambda' \\ & m
\end{pmatrix}z\right)=e\left(\frac{\lambda'}{m}P(r_a)\right)q^{\frac{P(r_a)}{m}}\sum_{n\geq 0}c_{r_a}(n)e\left(\frac{\lambda'}{m}n\right)q^{\frac{n}{m}},
$$
we obtain by Proposition \ref{slash} 
\begin{multline*}
\left(H_{r,m,t}|_k A\right)(z)=\frac{\zeta}{m}\sum_{\lambda\pmod{m}}e\left(\frac{\lambda}{m}\left(d^2 P(r_a)-P(r)-t\right)-\frac{\lambda'}{m} P(r_a)\right) H_{r_a}\left(\begin{pmatrix}
1 &  \lambda' \\ & m
\end{pmatrix}z\right)\\
=\frac{\zeta}{m}\sum_{\lambda\pmod{m}}e\left(\frac{\lambda}{m}\left(d^2 P(r_a)-P(r)-t\right)\right)q^{\frac{P(r_a)}{m}}\sum_{n\geq 0}c_{r_a}(n)e\left(\frac{\lambda'}{m}n\right)q^{\frac{n}{m}}\\
=\frac{\zeta}{m}q^{\frac{P(r_a)}{m}}\sum_{n\geq 0}c_{r_a}(n)e\left(\frac{dbn}{m}\right)\sum_{\lambda\pmod{m}}e\left(\frac{\lambda}{m}\left(d^2 (P(r_a)+n)-P(r)-t\right)\right)q^{\frac{n}{m}},
\end{multline*}
since $\lambda'\equiv db+d^2\lambda\pmod{m}$.\\

Now assume that $n$ is the smallest nonnegative integer with $d^2 (P(r_a)+n)- P(r)\equiv t\pmod{m}$. Then we have
$$
\left(H_{r,m,t}|_k A\right)(z)=\zeta_2 c_{r_a}(n) q^{\frac{P(r_a)+n}{m}}\left(1+O\left(q^{\frac{1}{m}}\right)\right)
$$
with $\zeta_2:=e\left(\frac{dbn}{m}\right)\zeta $.\\

Suppose that $H_{r,m,t}\equiv 0\pmod{p}$. Then $$
(p^{-1}H_{r,m,t})^{24Nm}\Delta^j
\in M_{24Nmk+12j}(\Gamma_1(24Nm))\cap \Z[\zeta_{24Nm}][q].$$
The $q$-expansion principle from Corollaire 3.12 of \cite{dr}, Ch.~VII states that if 
$f$ is a modular form of weight $\kappa$ for $\Gamma_1(N)$ whose Fourier coefficients at $i\infty$ lie in $\Z[\zeta_N]$, then for any $A\in\Gamma_0(N)$, also $f|_\kappa A$ has Fourier coefficients in $\Z[\zeta_N]$ (see also Corollary 5.3 of \cite{ra}).
Thus it follows that
$$
\left((p^{-1}H_{r,m,t})^{24Nm}\Delta^j\right)|_{24Nmk+12j} A
\in\Z[\zeta_{24Nm}][q]
$$
for every $A\in\Gamma_0(24Nm)$.
By the above computation we have
\begin{multline*}
\left(\left((p^{-1}H_{r,m,t})^{24Nm}\Delta^j\right)|_{24Nmk+12j} A\right)
=p^{-24Nm}(\left(H_{r,m,t}|A\right)(z))^{24Nm}\Delta(z)^j\\
=\left(\frac{c_{r_a}(n)}{p}\right)^{24Nm}q^{24N(P(r_a)+n)+j}\left(1+O(q)\right) \in \Z[\zeta_{24Nm}][q].
\end{multline*}
This can only hold if $p$ divides $c_{r_a}(n)$.
\end{proof}


\begin{thebibliography}{99}

\bibitem{ab} S. Ahlgren and M. Boylan, {\it Odd coefficients of weakly holomorphic modular forms}, Math. Res. Lett. {\bf 15} (2008), no. 3, 409--418. 

\bibitem{ak} S. Ahlgren, B. Kim, {\it Mock Theta Functions and Weakly Holomorphic Modular Forms Modulo $2$ and $3$}, Math. Proc. Cambridge Philos. Soc. {\bf 158} (2015), no. 1, 111--129.

\bibitem{ao} S. Ahlgren, K. Ono, {\it Congruecne Properties for the Partition Function}, Proc. Natl. Acad. Sci. USA {\bf 98} (2001), no. 23, 12882--12884.
	
\bibitem{dr} P. Deligne, M. Rapoport, {\it Les sch\'emas de modules de courbes elliptiques}, Modular functions of one variable II, Lecture Notes in Math. {\bf 349}, Springer, Berlin (1973), 143--316.

\bibitem{go} B. Gordon, {\it On the parity of the Rogers-Ramanujan coefficients}, Partitions, $q$-series, and modular forms, Dev. Math. {\bf 23}, Springer, New York (2012), 83--93.

\bibitem{hh} M. Hirschhorn, {\it On the 2- and 4-dissections of the Rogers-Ramanujan functions}, Ramanujan J. {\bf 40} (2016), no. 2, 227--235.

\bibitem{ko} I. Kiming and J. Olsson, {\it Congruences like Ramanujan's for powers of the partition function}, Arch. Math. (Basel) {\bf 59} (1992), no. 4, 348--360.

\bibitem{ono} K. Ono, {\it Distribution of the partition function modulo $m$}, Ann. of Math. (2) {\bf 151} (2000), no. 1, 293--307.

\bibitem{ra} S. Radu, {\it A Proof of Subbarao's Conjecture}, J. Reine Angew. Math. {\bf 672} (2012), 161--175.

\bibitem{rob} S. Robins, {\it Generalized Dedekind $\eta$-Products}, The Rademacher legacy to mathematics, Contemp. Math {\bf 166}, Amer. Math. Soc., Providence, RI (1994), 119--128.

\bibitem{sch} B. Schoeneberg, {\it Elliptic modular functions: an introduction}, Springer, New York (1974).

\bibitem{shu} I. Schur, {\it Zur additiven Zahlentheorie}, Sitzungsber. Preuss. Akad. Wiss. Phys.-Math. Kl. (1926), 488--495.

\end{thebibliography}
\end{document}